\documentclass[10pt,oneside,a4paper]{amsart}
\usepackage{geometry}
\usepackage[english]{babel}
\usepackage{amsfonts}
\usepackage{amsmath}
\usepackage{amssymb}
\usepackage{graphicx}
\usepackage{caption}
\usepackage{subcaption}
\usepackage{amsthm}
\usepackage{eurosym}
\usepackage[dvipsnames,usenames]{color}
\usepackage{cite}
\usepackage{soul}
\usepackage{tikz}
\usepackage{pgf}
\usepackage{lmodern} 
\usepackage[T1]{fontenc}
\title[Uniqueness properties for discrete equations and Carleman estimates]{\textbf{Uniqueness properties for discrete equations and Carleman estimates}}
\author{Aingeru Fern\'andez-Bertolin}
\author{Luis Vega}
\date{\today}
\address{A. Fern\'andez-Bertolin: Departamento de Matem\'aticas,  Universidad del Pa\'is Vasco UPV/EHU, apartado 644, 48080, Bilbao, Spain}
\email{aingeru.fernandez@ehu.eus}
\address{L. Vega: Departamento de Matem\'aticas,  Universidad del Pa\'is Vasco UPV/EHU, apartado 644, 48080, Bilbao, Spain
\newline
\null\hspace{1.95cm}Basque Center for Applied Mathematics BCAM, Alameda de Mazarredo 14, 48009 Bilbao, Spain}
\email{luis.vega@ehu.eus}
\usepackage{hyperref} 
\newtheorem{theorem}{Theorem}[section]

\newtheorem{lemma}{Lemma}[section]
\newtheorem{corollary}{Corollary}[section]
\theoremstyle{remark}
\newtheorem{remark}{Remark}[section]

\newcommand{\sumzd}{\sum_{j\in\mathbb{Z}^d}}

\begin{document}
%

%
\maketitle

\begin{abstract}
Using Carleman estimates, we give a lower bound for solutions to the discrete Schr\"odinger equation in both dynamic and stationary settings that allows us to prove uniqueness results, under some assumptions on the decay of the solutions. 
\end{abstract}

\section{Introduction}
The aim of this paper is to continue the study started in \cite{fb1,fb2} to prove uniqueness properties for functions $u\in C^1([0,1],\ell^2(\mathbb{Z}^d))$ which satisfy the property 
\begin{equation}\label{schr}
\left|i\partial_t u_j+\Delta_d u_j\right|\le |V_j u_j|,\ \ t\in[0,1],\ j\in\mathbb{Z}^d.
\end{equation}
with bounded potential $V$, under the assumptions that the function $u$ has fast decay at times $t=0$ and $t=1$. Here $\Delta_d$ stands for the discrete Laplace operator
\[
\Delta_du_j=\sum_{k=1}^d(u_{j+e_k}+u_{j-e_k}-2u_j),\ \ j\in\mathbb{Z}^d,
\]
where $e_k$ is the standard basis of $\mathbb{R}^d$.

In particular, all the results we give can be written in terms of solutions to the discrete Schr\"odinger equation
\begin{equation}
i\partial_t u_j+\Delta_d u_j+V_j u_j=0,\ \ t\in[0,1],\ j\in\mathbb{Z}^d.
\end{equation}

In the continuous case, these results are related to the Hardy uncertainty principle for the Fourier transform:
\[\begin{aligned}
&|f(x)|\le C e^{-|x|^2/\beta^2},\ \ |\hat{f}(\xi)|\le C e^{-4|\xi|^2/\alpha^2},\ \text{and }1/\alpha\beta >1/4\Longrightarrow f\equiv0.\\&
\text{If }1/\alpha\beta=\frac14,\text{ then } f(x)=ce^{-|x|^2/\beta^2}.
\end{aligned}\]

The relation comes from the fact that basically the solution to the free Schr\"odinger equation, $i\partial_t u+\Delta u=0$, has the same size as the Fourier transform of an appropriately modulated initial datum, so then the Hardy uncertainty principle can be stated, in an $L^2$ setting, as follows:

\[
\|e^{\alpha |x|^2}u(0)\|_{L^2(\mathbb{R}^d)}+\|e^{\beta |x|^2}u(1)\|_{L^2(\mathbb{R}^d)}<+\infty,\ \ \alpha\beta>\frac{1}{16} \Rightarrow u\equiv0.
\]

The classical proof of the Hardy uncertainty principle is based on complex analysis arguments (Phragm\'en-Lindel\"of principle and properties of entire functions), but in the dynamical context there is a series of papers, \cite{cekpv, ekpv1,ekpv2,ekpv3,ekpv4}, where the authors prove the Hardy uncertainty principle using real variable methods. Furthermore, not only do they prove their results for the free evolution, but they also include a potential term $Vu$ to the Schr\"odinger equation, under some size constraints for the potential $V$ but without any regularity assumption on it.  The main techniques in the proof of their results are log-convexity properties for solutions with Gaussian decay and Carleman estimates.

In the discrete setting, the first thing we have to understand is how to replace the Gaussian decay, so in \cite{fb1} we give an analogous version of the Hardy uncertainty principle by using complex analysis arguments that suggests that the discrete version of the Gaussian we have to consider is the product of modified Bessel functions, given by the following integral representation,
\[
I_m(x)=\frac{1}{\pi}\int_0^\pi e^{z\cos\theta}\cos(m\theta), \ m\in\mathbb{Z}.
\]

This product of modified Bessel functions appears naturally if we understand the Gaussian as the minimizer of the Heisenberg uncertainty principle. When we take discrete versions of the position and momentum operators to give a discrete Heisenberg principle, it turns out that the minimizer is precisely the product of modified Bessel functions. On the other hand, we can also understand the Gaussian as the fundamental solution to the heat equation, and, again, the fundamental solution in the discrete setting is given in terms of modified Bessel functions, whose decay is, for $n$ large,
\begin{equation}\label{decaybessel}
I_n(z)\sim \frac{1}{\sqrt{2\pi n}}\left(\frac{ez}{2}\right)^n e^{-n\log n},
\end{equation}
and we see here that for $n$ large, $I_n(z)$ decays like $e^{-n\log n}$.

Once we had a discrete version of the Hardy uncertainty principle proved by complex analysis, we proved in \cite{fb2} some log-convexity properties for solutions to the discrete Schr\"odinger equation with discrete Gaussian decay, where not only did we understand the Gaussian decay as above, but we also used other discrete versions of the Gaussian function. Then, by using Carleman estimates we could only give a preliminary result, that said that a solution cannot decay faster than $e^{-\mu|j|^2}$ at two different times. Looking at the behavior of modified Bessel functions, it is clear that this is far from the sharp result. Actually, independently in \cite{jlmp} it has recently been proved, for bounded and real-valued potentials, that in the one-dimensional case a solution cannot decay faster than $e^{-\mu|j|\log|j|}$ for $\mu>\frac{3+\sqrt{3}}{2}$ .

In this paper we try a different approach, exploited in \cite{ekpv1,ekpvkdv} in the continuous setting for Schr\"odinger and KdV equations and also based on log-convexity properties and Carleman estimates, in order to improve the result in \cite{fb2}. The main difference comes from the fact that first we prove the following lower bound for the solution:
\begin{theorem}[Lower bound for solutions to Schr\"odinger equations]\label{thm11}
Let $u\in C^1([0,1]:\ell^2(\mathbb{Z}^d))$ satisfying \eqref{schr} be such that
\[
\int_0^1\sum_{j\in\mathbb{Z}^d}|u_j(t)|^2\,dt\le A^2,\ 
\int_{1/2-1/8}^{1/2+1/8}|u(0,t)|^2\,dt\ge1.
\]

Let $V$ be such that
\[
\|V\|_{\infty}=\sup_{t\in[0,1],j\in\mathbb{Z}^d}\{|V_j(t)|\}\le L,
\]
then there exists $R_0=R_0(d,A,L)>0$ and $c=c(d)$ such that for $R\ge R_0$ it follows that
\[
\lambda(R)\equiv \left(\int_{0}^{1}\sum_{R-2\le |j|\le R+1}|u(j,t)|^2\right)^{1/2}\ge c e^{-cR \log R}.
\]
\end{theorem}

Then, by the use of similar log-convexity properties to those proved in \cite{fb2} we deduce the following result:

\begin{theorem}[Uniqueness result]\label{thm12}
Let $u\in C^1([0,1]:\ell^2(\mathbb{Z}^d))$ satisfying \eqref{schr} with $V$ a bounded potential. Then there exists $\mu_0=\mu_0(d)$ such that if, for $\mu>\mu_0$
\[
\sum_{j\in\mathbb{Z}^d}e^{2\mu|j|\log(|j|+1)}\big(|u_j(0)|^2+|u_j(1)|^2\big)<+\infty,
\]
then $u\equiv0$. Furtermore, if $d=1$, then we can take $\mu_0=1$.
\end{theorem}

Notice that in the one-dimensional setting this result agrees with the one in \cite{jlmp}, and it gives an improvement on the constant $\mu_0$ that leads to the sharp result with this rate of decay, as conjectured in \cite{jlmp} and proved for time-independent potential with compact support. Furthermore, the good behavior of the function $e^{\alpha|j|^2}$ in $\mathbb{Z}^d$ allows us to give a similar result in higher dimensions. In higher dimensions, unfortunately this result does not provide a sharp result.
 
However, the approach we use here is not suitable if one wants to relate this result to the continuous Hardy uncertainty principle. In order to do that, we should include the mesh step of the lattice, typically denoted by $h$,  and study what happens when $h$ tends to zero. In that case, the relevant region that relates the discrete and continuous results is of the type $|j|h\lesssim1$. Now, the Carleman inequality (see Lemma \ref{lem21} below), which is the key element in this approach holds as long as one assumes that $\frac{|j|}{R}$ is bigger than some constant, with $R$ large enough. If we identify the role of $\frac{1}{R}$ with the role of $h$, we clearly see that this result is giving information in the region that is not related to the continuous setting, and, therefore, this result is purely discrete. In order to give a relation between the discrete and the continuous settings, the appropriate weights for the Carleman estimates and log-convexity properties should be related to the modified Bessel function, and, as in \cite{ekpv2}, some interior estimates on the gradient should be required.

On the other hand, all the arguments can be adapted to the stationary case, where now we consider functions $u\in\ell^2(\mathbb{Z}^d)$ such that
\begin{equation}\label{elliptic}
\left|\Delta_d u_j\right|\le V_j u_j,\ \ j\in\mathbb{Z}^d,
\end{equation}
so we also have a lower bound in the ring $\{R-2<|j|<R+1\}$ and a uniqueness result. However, one may expect to improve our results (see Corollary \ref{cor31}) although our methods can not give such an improvement. Furthermore, a lower bound for solutions to the continuous elliptic problem is proved in \cite{bk} but, instead of the ring $\{R-2<|j|<R+1\}$, the lower bound is attained in a ball of radius 1 centered at some point of the sphere of radius $R$. In that paper it is pointed out that such a lower bound is not known in the discrete setting and it is most likely to be false, since one can extend the identically 0 function in a ball in such a way that its discrete Laplacian is zero but the extended function is not zero. For the sake of completeness, we will construct an example of a $\ell^2(\mathbb{Z}^2)$ solution to $\Delta_d u_j+V_ju_j=0$ with bounded potential that  is 1 at the origin but vanishes on a ball of fixed radius centered at a sphere of radius $R$, for $R$ large enough. 

Nevertheless, the same method explained here can be used in the continuous setting, and as an easy application one can get uniqueness properties for solutions to the elliptic problem $\Delta u+Vu=0$, assuming that the potential is bounded and that a solution decays faster than $e^{-\eta |x|^{4/3}}$ for some $\eta>0$, which coincides with the sharp result proved in \cite{m}.

The paper is organized as follows: In Section 2 we prove the main results of the paper, the lower bound and the uniqueness property for the time-dependent problem \eqref{schr}. In Section 3 we see that, as opposed to the continuous case, using this method we cannot improve the rate of decay from the evolution problem to the stationary one, and we also give an example of a solution which shows that the behavior explained in \cite{bk} is not possible in the discrete setting.

\section{Uniqueness for solutions to discrete Scrh\"odinger equations}
Before proving the main results, we need a discrete Carleman inequality for discrete Schr\"odinger evolutions in the spirit of the method developed in \cite{ekpv1,ekpvkdv}. In the case of the continuous Schr\"odinger evolution, the condition in the Carleman parameter is $\alpha\ge cR^2$, which leads to the Gaussian decay as the sharp rate of decay. In the discrete setting, the analogous of this condition is $\alpha\ge cR\log R$, so in this case we should look for solutions with this rate of decay, instead of solutions with Gaussian decay.

In order to simplify the notation, we write the notation $\|\cdot\|_2=\|\cdot\|_{L^2([0,1],\ell^2(\mathbb{Z}^d))}$.

\begin{lemma}[Carleman inequality for Schr\"odinger evolutions]\label{lem21}
Let $\varphi:[0,1]\rightarrow\mathbb{R}$ be a smooth function, $\beta>0$ and $\gamma>\frac{\sqrt{d}}{2\beta}$. There exists $R_0=R_0(d,\|\varphi'||_\infty+\|\varphi''\|_\infty,\beta,\gamma)$ and $c=c(d,\|\varphi'\|_\infty+\|\varphi''\|_\infty)$ such that, if $R>R_0$, $\alpha\ge \gamma R\log R$ and $g\in C_0^1([0,1],\ell^2(\mathbb{Z}^d))$ has its support contained in the set
\[
\{(j,t):|j/R+\varphi(t) e_1|\ge \beta\}.
\]
then
\[\begin{aligned}
\sqrt{\sinh(2\alpha/R^2)}\sinh(2\alpha\beta/\sqrt{d}R)\|e^{\alpha\left|\frac{j}{R}+\varphi(t) e_1\right|^2}g\|_2\le c \|e^{\alpha\left|\frac{j}{R}+\varphi(t) e_1\right|^2}(i\partial_t+\Delta_d)g\|_2.
\end{aligned}\]
\end{lemma}
\begin{proof} Let $f_j=e^{\alpha\left|\frac{j}{R}+\varphi(t)e_1\right|^2}g_j$. If we write
\[
e^{\alpha\left|\frac{j}{R}+\varphi(t)e_1\right|^2}(i\partial_t+\Delta_d)g_j=Sf_j+Af_j,
\]
with $S$ and $A$ symmetric and skew-symmetric respectively, it turns out that ($\delta_{mn}$ denotes Kronecker's delta function)
\begin{equation}\label{Ssch}\begin{aligned}
Sf_j=& i\partial_t f_j-2d f_j+\sum_{k=1}^d\cosh\left(\frac{2\alpha}{R}\left(\frac{j_k+1/2}{R}+\varphi \delta_{1k}\right)\right)f_{j+e_k}
\\&+\sum_{k=1}^d\cosh\left(\frac{2\alpha}{R}\left(\frac{j_k-1/2}{R}+\varphi\delta_{1k}\right)\right)f_{j-e_k},
\end{aligned}
\end{equation}
\begin{equation}\label{Asch}\begin{aligned}
Af_j=& -2i\alpha \left(\frac{j_1}{R}+\varphi\right)\varphi' f_j-\sum_{k=1}^d\sinh\left(\frac{2\alpha}{R}\left(\frac{j_k+1/2}{R}+\varphi \delta_{1k}\right)\right)f_{j+e_k}
\\&+\sum_{k=1}^d\sinh\left(\frac{2\alpha}{R}\left(\frac{j_k-1/2}{R}+\varphi \delta_{1k}\right)\right)f_{j-e_k}.
\end{aligned}
\end{equation}

Moreover, an easy computation shows that
\[
\|e^{\alpha\left|\frac{j}{R}+\varphi(t) e_1\right|^2}(i\partial_t+\Delta_d)g\|_2^2=\langle Sf+Af,Sf+Af\rangle \ge \langle [S,A]f,f\rangle.
\]

After some calculations, we have that the commutator is given so that
\begin{equation}\label{Csch}\begin{aligned}
\langle [S,A]f,f\rangle =&4\sinh\left(\frac{2\alpha}{R^2}\right)\int \sum_{j\in\mathbb{Z}^d}\sum_{k=1}^d\sinh^2\left(\frac{2\alpha}{R}\left(\frac{j_k}{R}+\varphi \delta_{1k}\right)\right)|f_j|^2\\
&+4\sinh\left(\frac{2\alpha}{R^2}\right)\int \sum_{j\in\mathbb{Z}^d}\sum_{k=1}^d\left|\frac{f_{j+e_k}-f_{j-e_k}}{2}\right|^2\\&+2\alpha\int\sum_{j\in\mathbb{Z}^d}\left[\left(\frac{j_1}{R}+\varphi\right)\varphi''+(\varphi')^2\right]|f_j|^2.
\\&+\frac{8\alpha}{R}\int\sum_{j\in\mathbb{Z}^d}\varphi' \cosh\left(\frac{2\alpha}{R}\left(\frac{j_1+1/2}{R}+\varphi\right)\right)\Im(f_{j+e_1}\overline{f_j}).
\end{aligned}\end{equation}

We want to hide the third and fourth term in the last expression in a fraction of the positive terms. Let us fist focus on the first term. Using Cauchy-Schwarz inequality, we have
\[
\left|\int\sum_j\psi(j)\Im(f_{j+e_1}\overline{f_j})\right|\le \int\sum_j\frac{|\psi(j-e_1)+|\psi(j)|}{2}|f_j|^2,
\]
and, if $\psi(j)=\frac{8\alpha}{R}\varphi'\cosh\left(\frac{2\alpha}{R}\left(\frac{j_1+1/2}{R}+\varphi\right)\right)$,
\[
\frac{|\psi(j-e_1)+|\psi(j)|}{2}=\frac{8\alpha}{R}|\varphi'|\cosh\left(\frac{\alpha}{R^2}\right)\cosh\left(\frac{2\alpha}{R}\left|\frac{j_1}{R}+\varphi\right|\right).
\]

Hence, we can absorb the fourth term in the first one if we establish that
\[
\sum_{k=1}^d\sinh\left(\frac{2\alpha}{R^2}\right)\sinh^2\left(\frac{2\alpha}{R}\left(\frac{j_k}{R}+\varphi\right)\right)\ge \frac{8\alpha\|\varphi'\|_\infty}{R}\cosh\left(\frac{\alpha}{R^2}\right)\cosh\left(\frac{2\alpha}{R}\left|\frac{j_1}{R}+\varphi\right|\right),
\]
when $\left|\frac{j}{R}+\varphi e_1\right|\ge \beta.$ But, in this case, there exists $k$ such that $\left|\frac{j_k}{R}+\varphi \delta_{1k}\right|\ge \frac{\beta}{\sqrt{d}}.$ Thus, it is enough to prove that, if $b\ge0$, $a\ge \max(b,\mu/\sqrt{d})$, $t=\frac{2\alpha}{R}$, $\kappa=4\|\varphi'\|_\infty$,
\begin{equation}\label{4term}
\sinh\left(\frac{t}{R}\right)\sinh^2(at)\ge \kappa t\cosh\left(\frac{t}{2R}\right)\cosh(bt),
\end{equation}
when $R>R_0$ and $t\ge 2\gamma\log R$. First, if $t\ge R$, $R\ge \max\{\sqrt{d}/\beta,\log(8k)/a\}$,
\[
\sinh\left(\frac{t}{R}\right)\ge \cosh\left(\frac{t}{2R}\right),\ \ \sinh(at)\ge \cosh(at)/4\ge \cosh(bt)/4,\ \ \sinh(at)\ge 4\kappa t,
\]
and, therefore, \eqref{4term} is established. Now, if $2\gamma\log R<t\le R$, $R\ge \sqrt{d}/\beta$,
\[
\sinh(at)\ge \cosh(at)/2\ge \cosh(bt)\cosh\left(\frac{t}{2R}\right)/(2e),\ \sinh(at)\ge e^{\beta t/\sqrt{d}}/4,\ \ \sinh\left(\frac{t}{R}\right)\ge \frac{t}{R},
\]
so we conclude
\[\sinh\left(\frac{t}{R}\right)\sinh^2(at)\ge \frac{e^{\beta t/\sqrt{d}}}{8 eR} t \cosh\left(\frac{t}{2R}\right)\cosh(bt).\]

It remains to see that $ \frac{exp\{\beta t/\sqrt{d}\}}{8 eR}\ge \kappa$. But, since $t>2\gamma\log R$,
\[
\frac{e^{\beta\mu t/\sqrt{d}}}{8 eR}\ge \frac{R^{\frac{2\beta\gamma}{\sqrt{d}}-1}}{8e},
\]
which, by the condition $\gamma>\frac{\sqrt{d}}{2\beta}$, is bigger than $\kappa$ for $R$ large enough depending on $\|\varphi'\|_\infty$, $\beta,\gamma$ and the dimension.

For the third term, using the same reasoning we have to prove that
\begin{equation}\label{3term}
\sinh\left(\frac{t}{R}\right)\sinh^2(at)\ge \kappa tRb,
\end{equation}
where the parameters $a,b,t$ are defined in the same way as above, and $\kappa=\|\varphi''\|_\infty.$ For $t\ge R$, $R/\log R\ge 2\sqrt{d}/\beta$,
\[
\sinh\left(\frac{t}{R}\right)\ge \frac{e}{4},\ \ \sinh(at)\ge bt,\ \sinh(at)\ge \sinh\left(\frac{\beta R}{\sqrt{d}}\right)\ge R^2,
\]
and \eqref{3term} is established if $R\ge \frac{\kappa 4}{e}$. Now, if $2\gamma \log R<t\le R$, $R\ge \sqrt{d}/\beta$
\[
\sinh\left(\frac{t}{R}\right)\ge \frac{t}{R},\ \sinh(at)\ge \frac{e^{bt}}{4},\ \sinh(at)\ge \frac{e^{\beta t/\sqrt{d}}}{4},
\]
and this implies that, for any $\epsilon>0$,
\[\sinh\left(\frac{t}{R}\right)\sinh^2(at)\ge \frac{e^{(2-\epsilon)\beta t/\sqrt{d}}e^{\epsilon bt}}{16R} t.\]

We want this quantity to be bigger than $\kappa t R b$. Using $e^{x}\ge x$ we reduce this to prove that
\begin{equation}\label{finaleq}
t e^{(2-\epsilon)\beta t/\sqrt{d}}\ge 16\kappa R^2/ \epsilon.
\end{equation}

Finally, thanks to the fact that $t>2\gamma\log R$, we have that for any $\gamma>\frac{\sqrt{d}}{2\beta}$ we can find $\epsilon>0$ such that \eqref{finaleq} holds for $R$ large enough, depending on the allowed parameters.   

Once that we have absorbed the third and fourth term in \eqref{Csch}, the inequality holds by using that $\left|\frac{j_k}{R}+\varphi \delta_{1k}\right|\ge \frac{\beta}{\sqrt{d}}$ for some $k=1,\dots,d$.
\end{proof}

Now we are going to use this lemma to prove the lower bound for a nonzero solution to the discrete Schr\"odinger equation. We recall that the fundamental solution to this equation is related to the Bessel function, $J_n(z)=I_n(-iz)i^{n}$, whose decay, for fixed $z$, is of the type $e^{-n\log n}$ when $n\rightarrow+\infty$, as in formula \eqref{decaybessel}.

\begin{proof}[Proof of Theorem \ref{thm11}] We define, for $\epsilon>0$ fixed, the $C^\infty(\mathbb{R}^d)$ cut-off functions $\theta^R(x),\mu(x)$ $(0\le \theta^R,\mu\le 1)$ and the $C^\infty([0,1])$ function $\varphi$ $(0\le \varphi\le 2+\epsilon^{-1})$ in the following way.
\begin{equation}\label{cutoff}
\theta^R(x)=\left\{\begin{array}{ll}1,&|x|\le R-1,\\0,&|x|\ge R,\end{array}\right.\ \ \mu(x)=\left\{\begin{array}{ll}1,&|x|\ge \epsilon^{-1}+1,\\0,&|x|\le \epsilon^{-1},\end{array}\right.\ \ \varphi(t)=\left\{\begin{array}{ll}2+\epsilon^{-1},&t\in[\frac12-\frac18,\frac12+\frac18],\\0,&t\in [0,\frac14]\cup[\frac34,1].\end{array}\right.
\end{equation}

\begin{remark}
The value of $\epsilon$ does not play an important role in the proof of Theorem \ref{thm11}. However, we use it because it will be crucial in order to quantify the best rate of decay we can give in order to achieve uniqueness in Theorem \ref{thm12} in the one-dimensional case.
\end{remark}

We are going to apply the previous lemma to
\[
g_j(t)=\theta_j^R\mu\left(\frac{j}{R}+\varphi(t)e_1\right)u_j(t),
\]
where $\theta_j^R=\theta^R(j)$. Notice that the evolution of $g$ is given by the expression
\begin{equation}\label{c}\begin{aligned}
(i\partial_t&+\Delta_d)g_j=i\varphi' \theta_j^R\partial_{x_1}\mu\left(\frac{j}{R}+\varphi e_1\right)u_j+\theta_j^R\mu\left(\frac{j}{R}+\varphi e_1\right)(i\partial_t u_j+\Delta_d u_j)\\&+\sum_{k=1}^d\left[\theta_{j+e_k}^R\left(\mu\left(\frac{j+e_k}{R}+\varphi e_1\right)-\mu\left(\frac{j}{R}+\varphi e_1\right)\right)+\mu\left(\frac{j}{R}+\varphi e_1\right)\left(\theta_{j+e_k}^R-\theta_{j}^{R}\right)\right]u_{j+e_k}\\&+\sum_{k=1}^d\left[\theta_{j-e_k}^R\left(\mu\left(\frac{j-e_k}{R}+\varphi e_1\right)-\mu\left(\frac{j}{R}+\varphi e_1\right)\right)+\mu\left(\frac{j}{R}+\varphi e_1\right)\left(\theta_{j-e_k}^R-\theta_{j}^{R}\right)\right]u_{j-e_k}.\end{aligned}\end{equation}

Thus, by paying with a dimensional constant $c_d$ and using \eqref{schr} we have
\begin{equation}\label{evolg}\begin{aligned}
\sqrt{\sinh(2\alpha/R^2)}&\sinh(2\alpha/\epsilon\sqrt{d}R)\|e^{\alpha\left|\frac{j}{R}+\varphi e_1\right|^2}g\|_2\le \|e^{\alpha\left|\frac{j}{R}+\varphi e_1\right|^2}(i\partial_t+\Delta_d)g\|_2\\ \le& L\|e^{\alpha\left|\frac{j}{R}+\varphi e_1\right|^2}g\|_2+c_d\left(\int_0^1\sum_{j\in\mathbb{Z}^d}e^{2\alpha\left|\frac{j}{R}+\varphi e_1\right|^2}\left|\partial_{x_1}\mu\left(\frac{j}{R}+\varphi e_1\right)\right|^2|u_{j}|^2dt\right)^{1/2}
\\&+c_d\left(\int_0^1\sum_{j\in\mathbb{Z}^d}\sum_{k=1}^de^{2\alpha\left|\frac{j}{R}+\varphi e_1\right|^2}\left|\mu\left(\frac{j+e_k}{R}+\varphi e_1\right)-\mu\left(\frac{j}{R}+\varphi e_1\right)\right|^2|u_{j+e_k}|^2dt\right)^{1/2}
\\&+c_d\left(\int_0^1\sum_{j\in\mathbb{Z}^d}\sum_{k=1}^de^{2\alpha\left|\frac{j}{R}+\varphi e_1\right|^2}\left|\mu\left(\frac{j-e_k}{R}+\varphi e_1\right)-\mu\left(\frac{j}{R}+\varphi e_1\right)\right|^2|u_{j-e_k}|^2dt\right)^{1/2}
\\&+c_d\left(\int_0^1\sum_{j\in\mathbb{Z}^d}\sum_{k=1}^de^{2\alpha\left|\frac{j}{R}+\varphi e_1\right|^2}\left|\theta^R_{j+e_k}-\theta^R_j\right|^2|u_{j+e_k}|^2dt\right)^{1/2}
\\&+c_d\left(\int_0^1\sum_{j\in\mathbb{Z}^d}\sum_{k=1}^de^{2\alpha\left|\frac{j}{R}+\varphi e_1\right|^2}\left|\theta^R_{j-e_k}-\theta^R_j\right|^2|u_{j-e_k}|^2dt\right)^{1/2}.
\end{aligned}\end{equation}

Now we study carefully the support of each term and we finish the proof taking $\alpha=cR\log R$ with $c=c(d,\epsilon)$ a constant satisfying the statement of the Carleman inequality. Indeed, with this choice of $\alpha$ the product of sinh functions in the left-hand side takes the form, for $R$ large,
\begin{equation}\label{beha}
\sqrt{2c\log R}R^{\frac{2c}{\epsilon\sqrt{d}}-\frac12},
\end{equation}
and, since $\frac{2c}{\epsilon\sqrt{d}}>1$ by Lemma \ref{lem21}, this grows with $R$, so we can absorb the term that comes from the potential taking $R$ large enough, depending on $L$. 

Now, by the definition of $\theta^R$ and $\mu$, we see that if $j=0$ and $t\in[1/2-1/8,1/2+1/8]$ then $\left|\frac{j}{R}+\varphi e_1\right|=2+\epsilon^{-1},$ so the cut-off functions are 1 and $g_0(t)=u_0(t)$. This allows us to bound the left-hand side of the Carleman inequality of the lemma by
\[
\|e^{\alpha\left|\frac{j}{R}+\varphi e_1\right|^2}g\|_2\ge e^{(2+\epsilon^{-1})^2\alpha},
\]
since $\int_{1/2-1/8}^{1/2+1/8}|u_0(t)|^2\ge 1$. 

On the other hand, we can use again the support of the cut-off functions to bound each term of the right-hand side in  \eqref{evolg}. In the first term, the one which involves a derivative of the function $\mu$, we see that in its support $\left|\frac{j}{R}+\varphi e_1\right|\le \epsilon^{-1}+1.$ 

For the terms involving the difference of $\theta$ functions, for each $k\in\{1,\dots,d\}$ we need to compute the coefficients where $\theta^R_{j\pm e_k}\ne \theta^R_j$, that is, we have to distinguish three different cases. First, when $\theta_{j\pm e_k}^R=1$, but  $\theta_{j}^R\ne1$, then when $\theta_{j\pm e_k}^R=0$, but  $\theta_{j}^R\ne0$, and finally when $0<\theta_{j\pm e_k}^R<1$, where since $|j\pm e_k|\ne |j|$ we have that $\theta^R_{j\pm e_k}\ne \theta^R_j$. After doing a shift in each term, it is easy to see that the difference is not zero in a region included in $\{R-2<|j|<R+1\}$, and we have that $\left| \frac{j}{R}+\varphi e_1\right|\le 3+\epsilon^{-1}+\frac{1}{R}$. Observe that when we bound $\left|\theta^R_{j\pm e_k}-\theta^R_j\right|\le 2$ for those that the difference does not vanish,we get the term $\lambda(R)$ of the statement. In this case, we also need to verify that $\mu(j/R+\varphi e_1)\ne0$. Indeed, if $t\in[\frac12-\frac18,\frac12+\frac18]$, and 
\[
\left|\frac{j}{R}+\varphi e_1\right|\ge 2+\epsilon^{-1} -1+\frac{1}{R}>\epsilon^{-1}+1.
\]

Finally, we have to see what happens with the terms involving the difference of $\mu$. In this case, we will see that if $\left| \frac{j}{R}+\varphi e_1\right|\ge \epsilon^{-1}+1+\frac{1}{R}$, then the function $\mu$ takes the same value at both points, so the difference is 0. Indeed, in this case we have that $\mu\left(\frac{j}{R}+\varphi e_1\right)=1$ and $\left|\frac{j}{R}+\varphi e_1\right|-\frac{1}{R}>0$. 

Therefore,
\[
\left|\frac{j\pm e_k}{R}+\varphi e_1\right|\ge \left| \left|\frac{j}{R}+\varphi e_1\right|-\frac{1}{R}\right| \ge \epsilon^{-1}+1+\frac{1}{R}-\frac{1}{R}=\epsilon^{-1}+1,
\]
so $\mu\left(\frac{j\pm e_k}{R}+\varphi e_1\right)=1$ as well. Gathering all these results we have, when $\alpha= cR\log R$ with $c>\frac{\epsilon\sqrt{d}}{2}$,
\begin{equation}\label{Scarl}\begin{aligned}
\sqrt{\sinh\left(\frac{2c\log R}{R}\right)}\sinh\left(\frac{2 c\log R}{\epsilon\sqrt{d}}\right)e^{cR\log R(2+\epsilon^{-1})^2}&\le c_{d,\epsilon}\left(e^{cR\log R\left(3+\epsilon^{-1}+\frac{1}{R}\right)^2}\lambda(R)\right.\\ &\left.+e^{cR\log R\left(\epsilon^{-1}+1+\frac{1}{R}\right)^2}A\right).
\end{aligned}\end{equation}

So for $R$ large enough, depending on $A$ (recall that before we showed that $R$ depends on $L$ as well) and $\epsilon$, which is a fixed number, we can absorb the second term in the right-hand side in the left-hand side and conclude
\[
1\le \sqrt{2c\log R}R^{\frac{2c}{\epsilon\sqrt{d}}-\frac{1}{2}}\le c_{d,\epsilon}e^{(5+2\epsilon^{-1})c R \log R+(6+2\epsilon^{-1})c\log R}\lambda(R),
\]
so
\begin{equation}\label{lambdaR}
\lambda(R)\ge c_{d,\epsilon}e^{-(5+2\epsilon^{-1})c R \log R-(6+2\epsilon^{-1})c\log R}
\end{equation}

\end{proof}

With this lower bound, we are able to prove the uniqueness result Theorem \ref{thm12}. For the sake of completeness, we recall the following result whose proof for solutions to the discrete Schr\"odinger equation can be found in \cite{fb2}. We recall that its proof is based on an abstract argument (see \cite[Lemma 2]{ekpv2}), and it is clear that the same proof works for functions satisfying \eqref{schr}.

\begin{lemma}\label{lem22}
Assume that $u$ satisfies \eqref{schr} where $V$ is a time-dependent bounded potential. Then, for $t\in[0,1]$ and $\beta\in\mathbb{R}^d$ we have
\begin{equation}\label{pesolin}
\sumzd e^{2\beta\cdot j}|u_j(t)|^2\le e^{C\|V\|_\infty}\sumzd e^{2\beta\cdot j}\big(|u_j(0)|^2+|u_j(1)|^2\big),
\end{equation}
where $C$ is independent of $\beta$, provided the left-hand side is finite.
\end{lemma}

From this result we can get a large variety of log-convexity properties for different weights, just by multiplication of \eqref{pesolin} with a proper function and integrating with respect to $\beta$, as it is explained in \cite{ekpv5}. For example, if we multiply it by the function $\exp(-2 \cosh(\beta/\mu)/e)$, we have that
\[
\int_\mathbb{R} e^{j\beta -2 \cosh(\beta/\mu)/e}\,d \beta =K_{\mu j}\left(\frac{2}{e}\right)\sim c \sqrt{2\mu |j|}e^{\mu |j|\log |j|+\mu |j| \mu \log \mu},
\] 
so the growth of this function is given by $e^{\mu |j|\log|j|}$ and we can adapt this to the multidimensional case to end up with a function that grows as $e^{\mu\|j\|_\star}$, where $\|j\|_\star=\sum_{k=1}^d|j_k|\log(|j_k|+1)$. On the other hand, it is easy to check that there is a dimensional constant $c_d$ such that
\[
\frac{\|j\|_\star}{c_d}\le |j|\log(|j|+1)\le c_d\|j\|_\star,
\]
hence, combining these two facts we have the following corollary.

\begin{corollary}\label{cor21}
Assume that $u=(u_j)_{j\in\mathbb{Z}^d}$ satisfies \eqref{schr} where $V$ is a time-dependent bounded potential. Then, for $\mu>0$ and $t\in[0,1]$, there are constants $c_0=c_0(d)$ and $c>0$ independent of $\mu$ and $t$ such that
\[
\sumzd e^{2\mu c_0|j|\log(|j|+1)}|u_j(t)|^2 \le e^{c\|V\|_\infty}\sumzd e^{2\mu |j|\log(|j|+1)}\big(|u_j(0)|^2+|u_j(1)|^2\big),
\]
provided that the right-hand side is finite.
\end{corollary}

\begin{proof}[Proof of Theorem \ref{thm12}]
If $u$ is not zero, by translation and dilation, we may assume that $u$ satisfies
\[
\int_{1/2-1/8}^{1/2+1/8}|u(0,t)|^2\,dt\ge1,
\]
so that we can apply the previous theorem to find a lower bound for $\lambda(R)$. On the other hand, using the previous corollary we have that
\[
\sup_{t\in[0,1]}\sum_{j\in{\mathbb{Z}^d}}|u(j,t)|^2e^{2\mu c_0|j|\log|j|}<+\infty.
\]

Hence, from this property for the solution, we prove the following upper bound for the quantity $\lambda(R)$ defined above,
\[
\lambda(R)\le ce^{-\mu c_0R\log R },
\]
while, by the previous theorem we know that $\lambda(R)\ge ce^{-cR\log R}$ for some $c$ depending on the dimension. Therefore, if $\mu$ is large enough (we need $\mu$ to be larger than the quotient $\frac{c}{c_0}$, so it only depends on the dimension) by letting $R\rightarrow\infty$ we reach a contradiction, so $u\equiv0.$

We finish this proof comparing this result with the one given in \cite{jlmp}, where the it is required $\mu>\frac{3+\sqrt{3}}{2}$. In our case, if we set $d=1$, \eqref{lambdaR} implies that 
\[\lambda(R)\ge e^{-(5+2\epsilon^{-1})cR\log R-(6+2\epsilon^{-1})c\log R},\]
for any $c>\frac{\epsilon}{2}$ and $\epsilon>0$ fixed. Moreover, in Corollary \ref{cor21}, in the one dimensional case we can take $c_0=1-\delta$ for any $\delta>0$, so, to get a contradiction in Theorem \ref{thm12}, if we set $c=\frac{\epsilon}{2}+\epsilon^2$ we need 
\[
\mu>\frac{(5+2\epsilon^{-1})(\epsilon/2+\epsilon^2)}{1-\delta}=\frac{1+9\epsilon/2+5\epsilon^2}{1-\delta}.
\] Thus, as long as $\mu>1$ we can find $\epsilon$ and $\delta$ such that we conclude $u\equiv0$.
\end{proof}

\begin{remark}\label{rem21}
 In higher dimensions, due to the log-convexity property we are using to derive Corollary \ref{cor21} we get that the best value of $\mu$ depends on the dimension, so the optimality of the constant remains to be proved.
\end{remark}

\section{Uniqueness for the stationary problem}
Now we turn to the stationary problem in \eqref{schr}, and we get uniqueness from Theorem \ref{thm12} as an immediate consequence, since, obviously, a stationary function satisfying \eqref{schr} satisfies $\left|\Delta_d u_j\right|\le |V_ju_j|$. Moreover, we can give  a stationary version of Lemma \ref{lem21}, just by taking $\varphi\equiv 3$, which makes the commutator positive.

As a result we have the following corollary:

\begin{corollary}\label{cor31}
Let $u\in \ell^2$ be such that $|\Delta_d u_j|\le |V_ju_j|$,
\[
\|u\|_2\le A,\ \ |u(0)|\ge 1,\ \text{and}\ \|V\|_\infty\le L.
\]

Then there exist $R_0=R_0(d,L,A)$ and $c=c(d)$ such that for $R\ge R_0$,
\[
\lambda(R)\equiv\left(\sum_{R-2\le|j|\le R+1}|u_j|^2\right)^{1/2}\ge c e^{-cR\log R}.
\]

Furthermore, there is $\eta_0$ depending on the dimension such that if
\[
\sum_{j\in\mathbb{Z}^d}e^{2\eta_0|j|\log(|j|+1)}|u_j|^2<+\infty,
\]
then $u\equiv0.$
\end{corollary}

\begin{remark}\label{rem31}
In this section, as opposed to the previous one, we consider $\|\cdot\|_2$ and $\|\cdot\|_\infty$ as the $\ell^2$ and $\ell^\infty$ norms of a sequence.
\end{remark}

As we have pointed out above, since the commutator in the Carleman inequality for the discrete Laplacian is positive, we do not need the extra assumption $\alpha\ge cR\log R$ required in the time-dependent case. Moreover, remember that in the application of the Carleman inequality, if we consider $\alpha=cR\log R$ the behavior of the left-hand side \eqref{beha} grows with $R$, so one could think that this result is most likely not sharp, and that we could consider $\alpha=c R\phi(R)$ with $1<\phi(R)<\log R$ and such that the product of the sinh functions in the Carleman inequality still grows with $R$ or we can make it independent of $R$ but as large as we want. However, in order to absorb the term that comes from the potential $V$ we require that
\[
\sinh(2\alpha/R^2)\sinh^2(2\alpha/\sqrt{d}R)\ge L^2,
\]
so, if $\alpha=cR\phi(R)$ with $1<\phi(R)<\log(R)$ we have that this is equivalent to
\[
\log(2c)+\log(\phi(R))+\frac{4c}{\sqrt{d}}\phi(R)-\log(R)\ge L^2,
\]
and, in order to satisfy this inequality, we need $\phi(R)$ to behave as $\log(R)$ when $R$ is large.

This result shows that the discrete and continuous settings exhibit different behaviors. On the one hand in the discrete setting the rate of decay is given by the exponential $e^{-c|j|\log(|j|+1)}$ while in the continuous case the solutions can decay faster. On the other hand, whereas in the discrete setting we get the same rate of the decay in our results, in the continuous setting the evolution and the stationary problems exhibit different rates of decay. In the case of Schr\"odinger, the decay is given in terms of Gaussians, while for the elliptic problem (see \cite{m}) the sharp rate of decay is of the type $e^{-c R^{4/3}}$. Actually using the same method we have used here in the discrete setting and Lemma \ref{lem31} below, one can prove a lower bound for solutions to the problem $(\Delta+V)u=0$ and the lower bound is precisely given in terms of $e^{-cR^{4/3}}$. 

\begin{lemma}\label{lem31}
The inequality
\[\begin{aligned}
\frac{\alpha^{3/2}}{R^2}\|e^{\alpha\left|\frac{x}{R}+3e_1\right|^2}g\|_2 + \frac{\alpha^{1/2}}{R}\|e^{\alpha\left|\frac{x}{R}+3e_1\right|^2}\nabla g\|_2\le  \|e^{\alpha\left|\frac{x}{R}+3 e_1\right|^2}\Delta g\|_2
\end{aligned}\]
holds, when  $g\in H^2(\mathbb{R}^d)$ has its (compact) support contained in the set $\{|x/R+3e_1|^2\ge 1\}$.
\end{lemma}

\begin{theorem}\label{thm31}
Let $u\in H^2$ be a solution to $\Delta u+Vu=0$ such that
\[
\|V\|_\infty\le L,\ \ \int_{|x|<1}|u(x)|^2\,dx\ge 1,\ \ \int(|u(x)|^2+|\nabla u(x)|^2)\,dx\le A^2,
\]
then, there is $c=c(L,A)$ such that
\[
\lambda(R)\equiv\left(\int_{R-1<|x|<R}(|u(x)|^2+|\nabla u(x)|^2)\,dx\right)^{1/2}\ge c e^{-c R^{4/3}}.
\]

Furthermore, there is $\eta_0$ depending on the dimension such that if
\[
\int_{\mathbb{R}^d}e^{2\eta_0|x|^{4/3}}|u(x,t)|^2\,dx<+\infty,
\]
then $u\equiv0.$
\end{theorem}

Moreover, in \cite{m} a counterexample of solution with this rate of decay is constructed, which implies that the decay is sharp. It would be interesting to know if the sharp decay in the discrete setting is $e^{-|j|\log(|j|+1)}$ by constructing a similar counterexample.

Nevertheless, in \cite{bk} the lower bound given in Theorem \ref{thm31} is improved in the sense that it is given in a smaller region. Instead of the ring $R-1<|x|<R$, they proved, for bounded solutions that for $x_0\in\mathbb{R}^d$ such that $|x_0|=R$,
\[
\max_{|x-x_0|\le 1}|u(x)|> c e^{-c \log R R^{4/3}},
\]
that is, they consider a ball of radius 1 centered at a point on the sphere of radius $R$ centered at the origin. This suggests that the discrete setting could have the same behavior. However, as was pointed out in \cite{k} this is not the case,  due to the existence of a counterexample (unpublished) of D. Jerison and C. Kenig who kindly sent us the details of it. For the sake of completeness we include below the counterexample appropriately modified for our needs.
More concretely we will construct a function $u$ satisfying the following
properties:

\begin{enumerate}
\item $u\in\ell^2(\mathbb{Z}^2)$,
\item $u(0,0)=1$,
\item $u(j,k)=0$ in the region $\Omega_R=\{|j|+|k-R|\le 2\},$ for some $R\in\mathbb{N}$ (large),
\item There exists a bounded potential $V$ such that $(\Delta_d+V)u=0$.
\end{enumerate}

To establish 4, one simply needs to set $V(j,k)=-\frac{\Delta_du(j,k)}{u(j,k)}$ and check that the potential defined this way is bounded. Notice that for this to make sense, we need to have $\Delta_du(j,k)=0$ whenever $u(j,k)=0$, which is the case of region $\Omega_R$. To do so, let
\[
\partial\Omega_R=\{(j,k)\in\mathbb{Z}^2:|j|+|k-R|=3\},
\] 
which corresponds to the region of those points outside $\Omega_R$ where at least one of the nearest neighbors $(j,k\pm1),\ (j\pm1,k)$ lies in $\Omega_R$. We recall that the discrete Laplacian of $u$ at the point $(j,k)$ only depends on the values of $u$ at $(j,k\pm1),\ (j\pm1,k)$ and at $(j,k)$. Then, one only needs to define $u$ properly on $\partial\Omega_R$ to have that both $u$ and $\Delta_du$ vanish on $\Omega_R$. This can be done in different ways, for example, defining $u$ as in Figure \ref{u vanish}.

\begin{figure}\centering
\begin{tikzpicture}
\node at (0,0) {0};\node at (1,0) {0};\node at (-1,0) {0};\node at (2,0) {0};\node at (-2,0) {0};\node at (0,1) {0};\node at (1,1) {0};\node at (1,-1) {0};\node at (0,2) {0};\node at (0,-2) {0};\node at (0,-1) {0};\node at (-1,1) {0};\node at (-1,-1) {0};
\node at (0,-3) {$2^{-m}$};\node at (0,3) {$2^{-m}$};
\node at (1,-2) {$-2^{-m-1}$};\node at (1,2) {$-2^{-m-1}$};\node at (-1,-2) {$-2^{-m-1}$};\node at (-1,2) {$-2^{-m-1}$};
\node at (2,-1) {$2^{-m-1}$};\node at (2,1) {$2^{-m-1}$};\node at (-2,-1) {$2^{-m-1}$};\node at (-2,1) {$2^{-m-1}$};
\node at (3,0) {$-2^{-m}$};\node at (-3,0) {$-2^{-m}$};
\end{tikzpicture}
\caption{Graphic representation of the region $\Omega_R\cup\partial\Omega_R$. In $\Omega_R$ $u$ vanishes and due to the construction of $u$ in $\partial\omega_R$ its discrete Laplacian at any point in the $\Omega_R$ also vanishes.}\label{u vanish}
\end{figure}

Since both $u$ and $\Delta_du$ vanish on $\Omega_R$, defining $V(j,k)=0$ in $\Omega_R$ gives that $u$ solves the equation in $\Omega_R$. It remains to define $u(j,k)=e^{-|j|-|k|}$ outside $\Omega_R\cup\partial\Omega_R$, and then we set $V(j,k)=-\frac{\Delta_du(j,k)}{u(j,k)}$.

It is clear that $V$ is well defined and that $u$, which is a solution to $(\Delta_d+V)u=0$ satisfies the first three properties of the list. It remains to prove that $V$ is bounded. Again, as $\Delta_du(j,k)$ only depends on the values of $u$ at $(j,k)$ and its nearest neighbors,
\[\begin{aligned}
|\Delta_du(j,k)|&=|2^{-|j+1|-|k|}+2^{-|j-1|-|k|}+2^{-|j|-|k+1|}+2^{|j|-|k-1|}-4\times2^{-|j|-|k|}|\\
&\le 4\times2^{-|j|-|k|+1}+4\times 2^{-|j|-|k|}=12 u(j,k),
\end{aligned}\]
for $(j,k)$ such that $|j|+|k-R|\ge 5$. For $(j,k)$ such that $3\le |j|+|k-R|\le 4$ we get bounds of the form
\[
|\Delta_du(j,k)|\le (c+2^{R-m})u(j,k),
\]
where $c$ is a universal constant, and $m$ is the number we use to define $u$ in $\partial\Omega_R$. Taking $m=R$ we see that not only is the potential bounded, but the bound does not depend on $R$.

\section{Acknowledgments} We would like to thank C. E. Kenig and E. Malinnikova for fruitful conversations. The work leading to this article was completed while both authors were in residence at the Mathematical Sciences Research Institue in Berkeley, California, during the Fall 2015 semester as Program Associate and Simons Visiting Professor, respectively. Both authors are partially supported by the projects  MTM2011-24054, IT641-13. The first author is also supported by the predoctoral grant BFI-2011-11 of the Basque Government and the IdEx 2016 Postdoctoral Program, and the second author by ERCEA Advanced Grant 2014 669689 - HADE and MINECO Severo Ochoa excellence accreditation SEV-2013-0323.  We are thankful to the referees for constructive comments that have improved the paper.

\end{document}